\documentclass{amsart}
\usepackage{amssymb}
\usepackage{graphicx}
\usepackage[compress]{cite}

\newtheorem{thm}{Theorem}
\newtheorem{cor}{Corollary}
\newtheorem{lem}{Lemma}

\newtheorem{rem}{Remark}

\newcommand{\A}{{\mathcal A}}

\newcommand{\U}{{\mathcal U}}
\newcommand{\es}{{\mathcal S}}

\newcommand{\D}{{\mathbb D}}

\newcommand{\real}{{\operatorname{Re}\,}}

\def\be{\begin{equation}}
\def\ee{\end{equation}}

\newcommand{\bee}{\begin{enumerate}}
\newcommand{\eee}{\end{enumerate}}

\newcommand{\blem}{\begin{lem}}
\newcommand{\elem}{\end{lem}}
\newcommand{\bthm}{\begin{thm}}
\newcommand{\ethm}{\end{thm}}
\newcommand{\bcor}{\begin{cor}}
\newcommand{\ecor}{\end{cor}}
\newcommand{\beg}{\begin{example}}
\newcommand{\eeg}{\end{example}}
\newcommand{\begs}{\begin{examples}}
\newcommand{\eegs}{\end{examples}}
\newcommand{\bdefe}{\begin{defin}}
\newcommand{\edefe}{\end{defin}}
\newcommand{\bprob}{\begin{prob}}
\newcommand{\eprob}{\end{prob}}
\newcommand{\bei}{\begin{itemize}}
\newcommand{\eei}{\end{itemize}}

\newcommand{\bcon}{\begin{conj}}
\newcommand{\econ}{\end{conj}}
\newcommand{\bcons}{\begin{conjs}}
\newcommand{\econs}{\end{conjs}}
\newcommand{\bprop}{\begin{propo}}
\newcommand{\eprop}{\end{propo}}
\newcommand{\br}{\begin{rem}}
\newcommand{\er}{\end{rem}}
\newcommand{\brs}{\begin{rems}}
\newcommand{\ers}{\end{rems}}
\newcommand{\bo}{\begin{obser}}
\newcommand{\eo}{\end{obser}}
\newcommand{\bos}{\begin{obsers}}
\newcommand{\eos}{\end{obsers}}
\newcommand{\bpf}{\begin{pf}}
\newcommand{\epf}{\end{pf}}
\newcommand{\ba}{\begin{array}}
\newcommand{\ea}{\end{array}}
\newcommand{\beq}{\begin{eqnarray}}
\newcommand{\beqq}{\begin{eqnarray*}}
\newcommand{\eeq}{\end{eqnarray}}
\newcommand{\eeqq}{\end{eqnarray*}}

\begin{document}
\bibliographystyle{amsplain}

\title[Hermitian Toepliz determinants fot univalent functions]{Hermitian Toepliz determinants for the class $\boldsymbol{\mathcal{S}}$ of univalent functions}

\author[M. Obradovi\'{c}]{Milutin Obradovi\'{c}}
\address{Department of Mathematics,
Faculty of Civil Engineering, University of Belgrade,
Bulevar Kralja Aleksandra 73, 11000, Belgrade, Serbia}
\email{obrad@grf.bg.ac.rs}

\author[N. Tuneski]{Nikola Tuneski$^1$}
\address{Department of Mathematics and Informatics, Faculty of Mechanical Engineering, Ss. Cyril and Methodius
University in Skopje, Karpo\v{s} II b.b., 1000 Skopje, Republic of North Macedonia.}
\email{nikola.tuneski@mf.edu.mk}







\begin{abstract}
Introducing a new method  we give sharp estimates of the  Hermitian Toepliz determinants of third order for the class  $\mathcal{S}$ of functions univalent in the unit disc. The new approach is also illustrated on some subclasses of the class $\es$.
\end{abstract}

\footnote{corresponding author; ORCID no.: 0000-0003-3889-0048.}

\maketitle

\section{Introduction}

Let $\mathcal{A}$ be the class of functions $f$ that are analytic  in the open unit disc $\D=\{z:|z|<1\}$ normalized such that $f(0)=f'(0)-1=0$ and let $\mathcal{S}\subset \A$ be the class of univalent functions in the unit disc $\D$ (functions that are analytic, one-on-one and onto).

\medskip

For functions $f\in \A$ of the form $f(z)=z+a_2z^2+a_3z^3+\cdots$ and positive integers $q$ and $n$, the  Toepliz matrix is defined by
\[
        T_{q,n}(f) = \left[
        \begin{array}{cccc}
        a_{n} & a_{n+1}& \ldots& a_{n+q-1}\\
        \overline{a}_{n+1}&a_{n}& \ldots& a_{n+q-2}\\
        \vdots&\vdots&~&\vdots \\
        \overline{a}_{n+q-1}& \overline{a}_{n+q-2}&\ldots&a_{n}\\
        \end{array}
        \right],
\]
where $\overline{a}_k = \overline{a_k}$. Thus, the second  Toepliz determinant is
\[
\left|T_{2,1}(f)\right|= 1-|a_2|^2
\]
and the third is
\begin{equation}\label{T3}
|H_{3,1}(f)| =  \left |
        \begin{array}{ccc}
        1 & a_2& a_3\\
        \overline{a}_2 & 1& a_2\\
        \overline{a}_3 & \overline{a}_2& 1\\
        \end{array}
        \right | = 2\, \real (a_2^2\overline{a}_3) -2|a_2|^2-|a_3|^2+1 .
\end{equation}

\medskip

The concept of  Toeplitz matrices plays an important role in functional analysis, applied mathematics as well as in physics and technical sciences (for more details see \cite{ye}).

\medskip

If $a_n$ is real, then the Toeplitz matrix $T_{q,n}(f)$ is an Hermitian one, i.e., it is equal to its conjugate transpose: $T_{q,n}(f) = \overline{[T_{q,n}(f)]^T}$. Determinants of  Hermitian matrices are real numbers.
Additionally, if $n=1$, the determinant $|T_{q,1}(f)|$ is rotationally invariant, i.e., for any real  $\theta$, the determinants $|T_{q,1}(f)|$ and $|T_{q,1}(f_{\theta})|$ of the Hermitian Toeplitz matrices of functions $f\in\A$ and $f_{\theta}(z):=e^{-i\theta}f(e^{i\theta}z)$ have same values.

\medskip

Recently, various problems of finding upper bounds, preferably sharp, of determinants of coefficients of classes of univalent functions, were rediscovered and attract significant interest. The highest focus is on the Hankel determinant and valuable references with overview of older results and the new ones are \cite{babaloa, opoola, Kowalczyk-18, kwon-1, kwon-2, lee, mishra, MONT-2019-1, MONT-2018-2, MONT-2018-1, MONT-2019-2,  selvaraj, shi, sokol, DTV-book, krishna, ye, zaprawa}.

\medskip

Naturally rises the question of finding lower and upper bound estimates of the determinant of the Hermitian Toeplitz matrices for the class of univalent functions and its subclasses. This problem was successfully solved sharp estimates in \cite{cudna} for the classes of starlike and convex functions of order $\alpha$, $0\le\alpha<1$, defined respectfully by
\[ \es^\ast(\alpha) = \left\{ f\in\A:\real\left[ \frac{zf'(z)}{f(z)} \right] >\alpha,\, z\in\D \right\}  \]
and
\[ \mathcal{C}(\alpha) = \left\{ f\in\A:\real\left[ 1+\frac{zf''(z)}{f'(z)} \right] >\alpha,\, z\in\D \right\}. \]

\medskip

For finding sharp estimates of the Hermitian Toeplitz determinant of second order it is enough to know sharp estimate for the second coefficient. The same question for the third order determinant turns out to be more complicated.

\medskip

In this paper we introduce new method for obtaining estimates of the Hermitian Toeplitz determinants of third order and receive sharp result for the general class $\es$ of univalent functions.

\medskip

We illustrate the new  method also on the class of convex functions, receiving the same sharp result as in \cite{cudna}.
In a similar manner we study classes
\[ \mathcal{U}_s(\lambda) = \left\{ f\in\mathcal{U}(\lambda): \frac{f(z)}{z}\prec \frac{1}{(1+z)(1+\lambda z)} \right\} \quad (0<\lambda\le1)\]
and
\[ \mathcal{G}(\delta) = \left\{ f\in\A:\real\left[ 1+\frac{zf''(z)}{f'(z)} \right] <1+\frac{\delta}{2},\, z\in\D \right\} \quad (0<\delta\le1), \]
where
\[ \mathcal{U}(\lambda) = \left\{ f\in\A:\left |\left (\frac{z}{f(z)} \right )^{2}f'(z)-1\right | < \lambda,\, z\in\D \right\} \quad (0<\lambda\le1)\] and "$\prec$" denotes the ususal subordination.
Class $\mathcal{U}(\lambda)$ is not included in the class of starlike functions $\es^\ast:=\es^\ast(0)$, nor vice versa (see \cite{OP-01,OP_2011}). Therefore estimates for $\es^\ast$ can not be transferred to the class $\U(\lambda)$. Sharp upper bound of the Hankel determinant of second and third order for the class $\U:=\U(1)$ are given in \cite{MONT-2018-2}.

\medskip

One can note that $\U=\U_s(1)$, since for all functions $f$ from $\U$, $\frac{z}{f(z)}\prec\frac{1}{(1+z)^2}$ (see \cite{OB-nashr}), while the general implication
\[ f\in\U(\lambda) \quad \Rightarrow \quad \frac{f(z)}{z}\prec \frac{1}{(1+z)(1+\lambda z)},\]
$0<\lambda<1$, was claimed in \cite{OPW-2016}, but  proven to be wrong  in \cite{lipon} by giving a counterexample.

\medskip

\section{Main results}

\medskip

\begin{thm}\label{th_S}
If $f\in\es$, then
\[-3\le |T_{2,1}(f)|\le 1 \quad\mbox{and}\quad -1\le |T_{3,1}(f)|\le 8. \]
All inequalities are sharp.
\end{thm}

\begin{proof}
From the Bieberbach's theorem (\cite{biebe}) we have $|a_2|\le2$ for all functions from $\es$ with Koebe's function $k(z)=\frac{z}{(1-z)^2}=\sum_{k=1}^\infty kz^k$ as an extremal one. Now both estimates for $|T_{2,1}(f)|$ directly follow, together with their sharpness.

\medskip

We continue with study of the third Toepliz determinant.

\medskip

Since for the class $\es$, $|a_3-a_2^2|\le1$ (see \cite[p.5]{DTV-book}), then
\begin{equation}\label{eq-re}
|a_2|^4+|a_3|^2-2\,\real(a_2^2\overline{a_3}) = |a_3-a_2^2|^2 \le1,
\end{equation}
and from here
\[ 2\,\real(a_2^2\overline{a_3}) \ge  |a_2|^4+|a_3|^2- 1. \]
Now, by using \eqref{T3} we have
\[ |T_{3,1}(f)| \ge (|a_2|^2-1)^2-1\ge-1, \]
which is sharp as the function $f_1(z)=\frac{z}{1-z+z^2} = z+z^2-z^4-\cdots$ shows.

\medskip
As for the upper bound of $|T_{3,1}(f)|$, from \eqref{T3}, by using that $\real\left(a_{2}^{2}a_{3}\right)\leq|a_{2}|^{2}|a_{3}|$, we obtain
\[
|T_{3,1}(f)| \le -|a_3|^2 + 2|a_2|^2|a_3| -2|a_2|^2 +1 =: \varphi(|a_3|),
\]
where
\[
\varphi(t) = -t^2+2|a_2|^2t-2|a_2|^2+1 \quad\mbox{and}\quad 0\le t=|a_3| \le 3.
\]
We need to find $\max\varphi(t)$ for $t\in[0,3]$.

\medskip

In that sense we have two cases.

\medskip
The first one is $0\le|a_2|^2\le3$, i.e., $0\le|a_2|\le\sqrt3$, when
\[ \max\varphi(t)=\varphi(|a_2|^2) = (|a_2|^2-1)^2\le4.\]

\medskip

The second case is $3\le|a_2|^2\le4$, i.e., $\sqrt3\le|a_2|\le\sqrt2$, when
\[ \max\varphi(t)=\varphi(3) = 4|a_2|^2-8\le8.\]

\medskip
Therefore, $\max\varphi(t)=8$, when $t\in[0,3]$.

\medskip

The result is sharp as the Koebe function $k(z)$ shows.
\end{proof}

\medskip

\begin{rem}\label{rem1}$ $
\begin{itemize}
  \item[$(i)$] The same result as in Theorem \ref{th_S} holds for the class $\es^\ast=\es^\ast(0)$ (see Corollary 1 and Corollary 3 from \cite{cudna}).
  \item[$(ii)$] The same result as in Theorem \ref{th_S} holds for the class $\U=\U(1)$ since $\U\subset\es$ and both extremal functions $f_1$ and $k$ belong to $\U$.
\end{itemize}
\end{rem}

\medskip

\begin{rem}\label{rem-2}
It is evident that for applying the method used in the proof of Theorem \ref{th_S} on other classes of univalent functions it is enough to know the sharp estimates of $|a_2|$, $|a_3|$ and $|a_3-a_2^2|$ and apply them on
\begin{equation}\label{TD2}
|T_{2,1}(f)|=1-|a_2|^2;
\end{equation}
and on
\begin{equation}\label{TD3}
 |T_{3,1}(f)| \le -|a_3|^2 + 2|a_2|^2|a_3| -2|a_2|^2 +1 =: \varphi(|a_3|),
\end{equation}
where $\varphi(t) = -t^2+2|a_2|^2t-2|a_2|^2+1$ and $t=|a_3|$.
\end{rem}

\medskip

In the sense of Remark \ref{rem-2}, for the class $\U_s(\lambda)$, using the sharp estimates
\begin{equation}\label{uc}
|a_2|\le 1+\lambda,\quad |a_3|\le 1+\lambda+\lambda^2\quad \mbox{and}\quad |a_3-a_2^2|\le\lambda,
\end{equation}
given in \cite{OPW-0}, \cite{OPW-1} and \cite{lipon}, we receive the following theorem.  Estimate $|a_2|\le 1+\lambda$ is sharp on the whole class $\U(\lambda)$.

\medskip

\begin{thm}
If $f\in\U(\lambda)$, then
\[-\lambda(2+\lambda)\le |T_{2,1}(f)|\le 1, \]
and if additionally $f\in\U_s(\lambda)$, then
\[
-\lambda^2\le |T_{3,1}(f)|\le
\left\{\begin{array}{cr}
1, & 0\le\lambda\le\lambda_0 \\
\lambda^2(1+\lambda)(3+\lambda), & \lambda_0\le\lambda\le1
\end{array}
\right.,
\]
where $\lambda_0=0.44762\ldots$ is the positive real root of the equation
\[\lambda^2(1+\lambda)(3+\lambda)-1=0.\]
All inequalities are sharp.
\end{thm}

\begin{proof}
The estimates of the second Hermitian Toepliz determinant follow directly from \eqref{TD2} and \eqref{uc} and they are sharp due to the functions $f_3(z)=z$ and
\[f_4(z)=\frac{z}{1-(1+\lambda)z+\lambda z^2} = \frac{z}{(1-z)(1-\lambda z)} = z+(1+\lambda)z^2+(1+\lambda+\lambda^2)z^3+\cdots.\]

\medskip

For the lower bound of the third Hermitian Toepliz determinant, from \eqref{TD3} and \eqref{uc},  we have
\[ |T_{3,1}(f)| \ge \left(|a_2|^2-1\right)^2 - \lambda^2 \ge -\lambda^2,\]
with sharpness for the function $f_2(z)=\frac{z}{1-z+\lambda z^2}=z+z^2+(1-\lambda)z^3+\cdots$. Function $f_2$ is analytic on $\D$ since $1-z+\lambda z^2$ equals zero on the unit disk only when $\lambda=0$ and $\lambda=-2$.

\medskip

For the upper bound of  $|T_{3,1}(f)|$ we consider two cases.

\medskip

In the first one, when $0\le|a_2|^2\le 1+\lambda+\lambda^2$, the vertex of the parabola $\varphi(t)$ is obtained for $t=|a_2|^2$ and lies in the range of $t=|a_3|$. So,
\[
\begin{split}
|T_{3,1}(f)|&\le \max\varphi(t) =\varphi(|a_2|^2) = \left( |a_2|^2-1 \right)^2\\
& \le
\left\{\begin{array}{cc}
1, & |a_2|^2\le2 \,\,\,(\Leftrightarrow 0<\lambda\le\frac{\sqrt5-1}{2}) \\
\lambda^2(1+\lambda)^2, & 2\le|a_2|^2\le 1+\lambda+\lambda^2 \,\,\,(\Leftrightarrow \frac{\sqrt5-1}{2}\le\lambda\le1)
\end{array}
\right..
\end{split}
\]

\medskip

Similarly, in the second case, $1+\lambda+\lambda^2 \le |a_2|^2\le (1+\lambda)^2$,  we have that the vertex lies on the right of the range of $t=|a_3|$. Thus
\[
|T_{3,1}(f)| \le \max\varphi(t) =\varphi(1+\lambda+\lambda^2) = \lambda^2(1+\lambda)(3+\lambda).
\]

\medskip

By using all these facts, we conclude that
\[
|T_{3,1}(f)| \le
\left\{\begin{array}{cc}
1, &  0<\lambda\le\lambda_0 \\
\lambda^2(1+\lambda)(3+\lambda), & \lambda_0 \le\lambda\le1
\end{array}
\right.,
\]
where $\lambda_0=0.44762\ldots$ is the positive real root of the equation
\[\lambda^2(1+\lambda)(3+\lambda)-1=0.\]

\medskip

The upper bound of the third order determinant is also sharp with extremal function $f_3$ when $0<\lambda\le\lambda_0$ and $f_4$ when $\lambda_0 \le\lambda\le1$.
\end{proof}

\medskip

For $\lambda=1$ we receive the following corollary with the same estimates as for the class $\es$ already discussed in Remark \ref{rem1}($ii$).

\medskip

\begin{cor}
If $f\in\U\equiv\U(1)$, then $-3\le |T_{2,1}(f)|\le 1$, and if $f\in\U_s\equiv\U_s(1)$, then $-1\le |T_{3,1}(f)|\le 8$. All inequalities are sharp.

\end{cor}

\medskip

We conclude with two more applictions of Remark \ref{rem-2}.

\begin{thm}
If $f\in\mathcal{C}:=\mathcal{C}(0)$, then $0\le |T_{3,1}(f)|\le1$. The estimate is sharp.
\end{thm}

\begin{proof}
For the class $\mathcal{C}$ of convex functions we know that
  \[|a_3-a_2^2|\le \frac13(1-|a_2|^2)\]
  (see \cite{trimble}). So, from \eqref{TD3} we have
\[  |T_{3,1}(f)| \ge \frac89 \left(1-|a_2|^2\right)^2  \ge 0. \]
The function $f_5(z)=\frac{z}{1-z}=z+z^2+z^3+\cdots$ shows that this result is sharp.

\medskip

On the other hand side, since $0\le |a_2|\le1=\max|a_3|$, we have
\[ |T_{3,1}(f)| \le \max\varphi(t) = \varphi(|a_2|^2) = (|a_2|^2-1)^2 \le 1,\]
with equality for $f_3(z)=z.$

\medskip

Therefore, $0\le |T_{3,1}(f)|\le1$ which is the same result as in Corollary 6 from \cite{cudna}.
\end{proof}

\medskip

\begin{thm}
If $f\in\mathcal{G}:=\mathcal{G}(1)$ we have sharp estimates $\frac12 \le |T_{3,1}(f)| \le 1$.
\end{thm}

\begin{proof}
For the class $\mathcal{G}$ we have
  \[ |a_2|\le \frac12,\quad |a_3|\le \frac16\quad \mbox{and}\quad |a_3-a_2^2|\le\frac14 \]
  (see \cite{OPW-0}). Then
  \[  |T_{3,1}(f)| \ge \left(1-|a_2|^2\right)^2 -\frac{1}{16} \ge \left(\frac34\right)^2 -\frac{1}{16}=\frac12. \]
The result is sharp as the function $f_6(z)=z-\frac12z^2$ shows.

\medskip

As for the upper bound, for $0\le |a_2|^2\le\frac16=\max|a_3|$ we have
\[ \max\varphi(t) = \varphi(|a_2|^2) = (|a_2|^2-1)^2 \le 1,\]
while for $\frac16\le |a_2|^2\le\frac14$,
\[ \max\varphi(t) = \varphi\left(\frac16\right) = \frac{35}{36}-\frac53|a_2|^2\le \frac{35}{36}-\frac53\cdot\frac16=\frac{25}{36},\]
which implies that $|T_{3,1}(f)| \le \max\varphi(t)=1 $ for $0\le t=|a_3|\le\frac16$ . The result is sharp for $f_3(z)=z$.

\medskip

This result can be easily generalized on the class $\mathcal{G}(\delta)$ using the sharp estimates require for the method given in \cite{OPW-0}.
\end{proof}

\end{document}